\documentclass{article}
\usepackage{graphicx}
\usepackage{amsmath}
\usepackage{amsfonts}
\usepackage[indent=10pt]{parskip}
\usepackage[english]{babel}
\usepackage{babel}
\usepackage{csquotes}
\usepackage{hyperref}
\usepackage{amsthm}
\usepackage[margin=1in]{geometry}
\usepackage{biblatex}
\newtheorem{theorem}{Theorem}

\newtheorem{lemma}[theorem]{Lemma}

\theoremstyle{definition}
\newtheorem{definition}{Definition}[section]
\newtheorem*{remark}{Remark}
\usepackage{biblatex} 
\usepackage{parskip} 
\DeclareMathOperator{\supp}{supp}
\DeclareMathOperator{\Lip}{Lip}

\newcommand{\mres}{\mathbin{\vrule height 1.6ex depth 0pt width
0.13ex\vrule height 0.13ex depth 0pt width 1.3ex}}
\title{A universal support theorem for 1-Wasserstein optimal transport}
\author{Ng Ze-An}
\date{}

\begin{document}
\maketitle
\begin{abstract}
   We show that every optimal transport plan for the $1$-Wasserstein distance in $\mathbb R^n$ is supported on a closed $n+1$ rectifiable set. In particular, the support of any optimal plan has Hausdorff dimension at most $n+1$.
\end{abstract}

\section{Introduction}
The optimal transport problem, first investigated by Monge and later extended by Kantorovich, asks for the most efficient way to transport a prescribed mass distribution onto a target distribution with respect to some cost functional. In its modern formulation, it takes the following form. 

Let $\mu, \nu$ be probability measures on $\mathbb R^n$. A \textit{coupling} $\pi$ of $\mu$ and $\nu$ is a probability measures $\pi$ on $\mathbb R^n \times \mathbb R^n$ whose marginals on the first and second factor equal $\mu$ and $\nu$ respectively - that is, $\pi(E \times \mathbb R^n) = \mu(E)$ and $\pi(\mathbb R^n \times E) = \nu(E)$ for all Borel subsets $E$ of $\mathbb R^n$.

Let $c: \mathbb R^n \times \mathbb R^n \to \mathbb R_+$ be some non-negative cost function. The Monge-Kantorovich optimal transport problem in its general form asks to minimize
$$\int_{\mathbb R^n \times \mathbb R^n} c(x, y) \, d\pi(x,y)$$
over all couplings $\pi$ of $\mu, \nu$. Such couplings $\pi$ are customarily denoted as \textit{transport plans}.

Of all cost functions, the most fundamental and well studied are the \text{Wasserstein} costs $c(x, y) := d(x, y)^p$, with $1 \leq p \leq \infty$. When $p = 1$, the cost is just the distance function, which was the original cost investigated by Monge in early formulations of the problem. The resulting metric on probability measures
\begin{equation}\label{W1}\mathcal W_1 (\mu, \nu) := \inf_{\pi \text{ a coupling of }\mu, \nu} \int_{\mathbb R^n \times \mathbb R^n} |x-y|  \,d\pi(x,y)\end{equation}
is known as the $1$-Wasserstein distance, and is well defined whenever $\mu, \nu$ lie in the space $\mathcal P_1 (\mathbb R^n)$ of probability measures $\eta$ whose first moment
$$\int_{\mathbb R^n} |x| \, d\eta(x)$$
is finite. We refer to the book \cite{Maggi_2023} by Maggi for an excellent overview of the general theory, including a detailed analysis of the $1$-Wasserstein transport problem.

In this note, we prove a support theorem for optimal plans for the optimal transport problem (\ref{W1}) with respect to the distance cost. 

Our main result, Theorem \ref{support theorem} is that the support of any optimal plan between measures in $\mathcal P_1(\mathbb R^n)$ is $n+1$ rectifiable, in the sense of geometric measure theory. This is a strong form of geometric regularity which implies, in particular that the support of any optimal plan has Hausdorff dimension at most $n+1$. Because transport plans live in $\mathbb R^n \times \mathbb R^n$, the Hausdorff dimension can a priori be up to $2n$. Thus the main result shows that optimal transport plans are supported on a much thinner set than the ambient space.

Due to Kantorovich duality, the analysis of optimal transport plans for the $1$-Wasserstein distance is intimately tied with the geometry of $1$-Lipschitz functions. The proof of the main Theorem \ref{support theorem} thus relies crucially on a theorem in the fine analysis of Lipschitz functions, Theorem \ref{stretch set theorem}, which may be of independent interest.

The rest of the paper is structured as follows. In Section 2, we make some preliminary definitions and fix notation. In Section 3, we state and prove the two main theorems \ref{support theorem} and \ref{stretch set theorem}. In Section 4, we give an example showing that the upper bound of $n+1$ on the Hausdorff dimension cannot be improved in general. In Section 5, we state some natural further questions.

\section{Preliminaries}

Given points $a, b \in \mathbb R^n$, we use $[a, b]$ to denote the line segment 
$$\{ta + (1-t)b : \, t \in [0, 1]\}.$$ 

We denote by $\Lip (u)$ the best Lipschitz constant of $u: \mathbb R^n \to \mathbb R$, the infimum of all $L > 0$ such that 
$$|u(x) - u(y)| \leq L|x-y|.$$

\begin{definition}
The \textbf{stretch set} of a Lipschitz function $u$ is the set
$$\mathbf \Gamma(u) := \left \{(x, y) \in \mathbb R^n \times \mathbb R^n:  \, |u(x) - u(y)| = \Lip(u) \,|x-y|\right \}$$
of pairs of points on which $u$ realizes its best Lipschitz constant. Note that the stretch set contains in particular all points on the diagonal $\Delta := \{(x, x): \, x \in \mathbb R^n\}$.
\end{definition}
We write
$$\Gamma(u) = \{x \in \mathbb R^n \, | \, (x, y) \in \mathbf \Gamma(f) \setminus \Delta\}$$
for the projection of $\mathbf \Gamma(f) \setminus \Delta$ to $\mathbb R^n$.

\begin{definition}Let $E$ be an arbitrary subset of $\mathbb R^d$, and $k < d$ a nonzero integer. $E$ is said to be \textbf{$k$-rectifiable} if there exist countably many Lipschitz maps $\psi_i: \mathbb R^k \to \mathbb R^d$ such that $E \subseteq \bigcup_i \psi_i (\mathbb R^k)$.
\end{definition}

\begin{remark}Note that we have used the stronger convention which requires the rectifiable set to be fully contained within the image of countably many Lipschitz maps, instead of allowing for an exceptional $\mathcal H^k$-null set to be omitted. As usual it will be enough to define the maps $\psi_i$ on arbitrary subsets of $\mathbb R^k$, as they can then be extended to Lipschitz maps on all of $\mathbb R^k$ by the Kirszbraun-Valentine Lipschitz extension theorem.\end{remark}
\section{Main Theorems}

\begin{theorem}[Support theorem]\label{support theorem}Let $\pi$ be an optimal transport plan between $\mu,\nu \in \mathcal P_1 (\mathbb R^n)$ for the $1$-Wasserstein distance on $\mathbb R^n$. Then the support of $\pi$ is contained within a closed $n+1$ rectifiable set, which can be taken to be the stretch set of any associated Kantorovich potential. In particular, the support of $\pi$ has Hausdorff dimension at most $n+1$. \end{theorem}

We recall that as a consequence of Kantorovich duality, to every optimal plan $\pi$ there is associated a $1$-Lipschitz \text{Kantorovich potential} $f: \mathbb R^n \to \mathbb R$ such that $\pi(\Gamma(f)) = 1$ (\cite{Maggi_2023}, Theorem 3.17). Theorem \ref{support theorem} is thus an immediate consequence of the following general theorem on Lipschitz functions.

\begin{theorem}[Rectifiability of the stretch set]\label{stretch set theorem}
Let $f: \mathbb R^n \to \mathbb R$ be a Lipschitz function. Then the stretch set $\mathbf \Gamma(f)$ is closed and $n+1$ rectifiable.
\end{theorem}

\begin{proof}[Proof of Theorem \ref{support theorem} assuming Theorem \ref{stretch set theorem}.]
Let $\pi$ be an optimal transport map for the $1$-Wasserstein distance. By Kantorovich duality, choose any Kantorovich potential $f$ such that $\pi(\mathbf \Gamma(f)) = 1$. Since $\mathbf \Gamma(f)$ is closed, we have that $\text{supp}(\pi) \subseteq \mathbf \Gamma(f)$. Both claims then follow immediately from Theorem \ref{stretch set theorem}.
\end{proof}

The rest of this section will thus be dedicated to proving Theorem \ref{stretch set theorem}. We first make an important pair of definitions.

\begin{definition}
Let $f: \mathbb R^n \to \mathbb R$ be a Lipschitz function. 

A \textbf{stretch ray} of $f$ is a non-degenerate line segment $[a, b]$ such that
$$\frac{|f(s) - f(t)|}{|s-t|} = \text{Lip}(f)$$
for all distinct points $s, t$ in $[a, b]$. 

An \textbf{oriented stretch ray} of $f$ is an ordered tuple $[[a, b \, ]]$ such that
$$f((1-t)a+tb) - f((1-s)a + sb) = \text{Lip}(f) \, (t-s) \,|b-a|$$
for all $s < t$ in $[0, 1]$.\end{definition}

We now present the proof of Theorem \ref{stretch set theorem}. For continuity of exposition, we will prefer to simply state key technical lemmas, deferring their proofs until after the main proof.

\begin{proof}[Proof of Theorem \ref{stretch set theorem}]
Let $f: \mathbb R^n \to \mathbb R$ be an arbitrary Lipschitz function, which after scaling we can assume without loss of generality is $1$-Lipschitz. The following lemma is narratively independent of the rest of the proof, so we prefer to state it first.
\begin{lemma}\label{closed} The stretch set $\mathbf \Gamma(f)$ is closed.\end{lemma}
Next we note that the stretch set can be written as the potentially uncountable union of stretch rays.

\begin{lemma}  \label{stretch ray decomposition} We have
\begin{equation} \mathbf \Gamma(f) = \Delta \cup  \bigcup_{\ell} \, \ell \times \ell \end{equation}
where $\ell$ runs over all stretch rays of $f$.\end{lemma}

Note that the union over stretch rays above is not a disjoint union, and in general contains a large amount of overlap.

Now we turn to producing the covering Lipschitz charts. Fix a dense, countable set $\{\mathbf H_i\}_{i \in I}$ of hyperplanes in the Grassmannian $\textbf{Gr}_{n-1, n}$ of $n-1$ dimensional hyperplanes through the origin in $\mathbb R^n$. 

For each $q\in \mathbb Q$, set $\mathbf H_{i, q} := \mathbf H_i + q \nu_i$, where for each $i \in I$, $\nu_i$ is an arbitrarily chosen unit normal to $\mathbf H_i$.

For each $i \in I$, set $ \Gamma_{i, q} :=  \Gamma(f) \cap \mathbf H_{i,q}$.

Then for each $m \in \mathbb Z_+$, define the set $ \Gamma_{i, q, m}$ as
$$\left \{x \in \Gamma_{i, q} \, : \,\exists e \in S^{n-1}  \text{ such that } \left [\left [x-\frac{e}{m}, x+\frac{e}{m} \right ] \right ] \text{ is an oriented stretch ray} \right \}.$$

Intuitively, $\Gamma_{i, q, m}$ is the set of points of stretch rays that intersect the hyperplane $\mathbf H_{i, q}$ and extend sufficiently long in both directions.

\begin{lemma}\label{unique direction} The unit vector $e_{i,q,m} (x) \in S^{n-1}$ above is unique for each $x \in \Gamma_{i,q,m}$.\end{lemma}

\begin{lemma}\label{nonempty intersection}Let $\ell$ be a stretch ray of $f$. Then there exists some $\Gamma_{i,q,m}$ such that $\ell \cap \Gamma_{i,q,m}$ is nonempty.
\end{lemma}

We now define the charts $\psi_{i, q, m, k}: \Gamma_{i, q, m} \times [-k, k]^2 \to  \mathbb R^n \times \mathbb R^n$ by
$$\psi_{i,q, m, k} \,  (x, s, t) := (x + se_{i,q,m}(x) , x + te_{i,q,m}(x))$$
where the new index $k$ runs over $\mathbb Z_+$.
\begin{lemma}\label{contains every stretch}The countable union 
\begin{equation}\label{big union}\bigcup_{i \in I} \bigcup_{q \in \mathbb Q}  \bigcup_{m \in \mathbb Z_+} \bigcup_{k \in \mathbb Z_+} \psi_{i,q,m,k} \, ( \Gamma_{i, q, m} \times [-k, k]^2)\end{equation}
contains $\ell \times \ell$ for every stretch ray $\ell$ of $f$. \end{lemma}

Combining Lemma \ref{stretch ray decomposition} and \ref{contains every stretch} we see that the union (\ref{big union}) contains the entire stretch set $\mathbf \Gamma(f)$, except possibly the diagonal $\Delta$. Thus along with the single diagonal chart $\psi_{\Delta}: \mathbb R^{n} \times \{0\} \to \mathbb R^n \times \mathbb R^n$ defined by
$$\psi_{\Delta}(x, 0) = (x, x)$$
which is easily seen to be Lipschitz, we obtain the desired countable chart covering of $\mathbf\Gamma(f)$.

It is left only to show that each chart $\psi_{i,q,m,k}$ is Lipschitz, which will follow from the following key Lipschitz estimate.

\begin{lemma}\label{lipschitz angle} Each function $e_{i,q,m}:  \Gamma_{i,q,m} \to  S^{n-1}$ is Lipschitz.
\end{lemma}

Using Lemma \ref{lipschitz angle}, we compute that each $\psi_{i,q,m,k}$ is Lipschitz as follows. Denoting by $L_e$ the Lipschitz constant of $e_{i,q,m}$ ,we have
\begin{align*}  & |\psi_{i,q,m,k} (x_1, s_1, t_1) - \psi_{i,q,m,k} (x_2, s_2, t_2)| \\
\leq & |x_1 - x_2| + |s_1 e_{i,q,m}(x_1) - s_2 e_{i,q,m} (x_2)| +  |t_1 e_{i,q,m}(x_1) - t_2 e_{i,q,m} (x_2)|\\
\leq & |x_1 - x_2| + 2k|e_{i,q,m}(x_1) - e_{i,q,m}(x_2)| + 2(|s_1 - s_2| + |t_1 - t_2|) \\
 \leq & (2kL_e + 1)|x_1 - x_2| + 2(|s_1 - s_2| + |t_1 - t_2|) \\
\leq & (2kL_e + 5) \,  \text{dist}((x_1, t_1, s_1), (x_2, t_2, s_2)).
\end{align*}
This concludes the proof of Theorem \ref{stretch set theorem}.
\end{proof}

We now prove the auxiliary lemmas above.

\begin{proof}[Proof of Lemma \ref{closed}]
Let $(x_i, y_i) \in \mathbf \Gamma(f)$ be a convergent sequence, converging to some $(x, y)$. For each $i$, we have
\begin{align*}|f(x) - f(y)|& \leq |f(x_i) - f(y_i)| + |f(x) - f(x_i)| + |f(y) - f(y_i)| \\ & \leq |x_i - y_i| + |x - x_i| + |y - y_i|\end{align*}
where in the last line we have used that $(x_i, y_i) \in \mathbf \Gamma(f)$ and the $1$-Lipschitz property of $f$. Likewise, 
$$|f(x) - f(y)| \geq |x_i - y_i| - |x - x_i| - |y - y_i|.$$
Taking the limit in $i$ above, we obtain
$$|f(x) - f(y)| = |x-y|$$
so that $(x, y) \in \mathbf \Gamma(f)$ as desired.
\end{proof} 

\begin{proof}[Proof of Lemma \ref{stretch ray decomposition}]
The containment $\Delta \cup \bigcup_{\ell} \ell \times \ell \subset \mathbf \Gamma(f)$ follows verbatim from the definition of $\mathbf \Gamma(f)$. For the reverse containment, let $(x, y) \in \mathbf \Gamma(f)$. If $(x, y) \in \Delta$, there is nothing to prove, otherwise let us show that $[x, y]$ is a stretch ray. Indeed, let $(a, b) \in [x, y]$. Without loss of generality, $a$ lies closer to $x$ than $b$. Then we have, by the triangle inequality

$$ |f(a) - f(b)| \geq |f(x) - f(y)| - |f(a) - f(x)| - |f(b) -f(y)| \geq |x-y| -|a-x| - |b-y| = |a-b|,$$
where the last line follows because the points $x, a, b, y$ are colinear and monotone along the segment $[x, y]$. Since $f$ is $1$-Lipschitz, we have
$$f(a) - f(b) = |a-b|,$$
which shows that $(a, b) \in \mathbf \Gamma(f)$ for all $a, b \in [x, y]$, so $[x, y]$ is a stretch ray. Thus $(x,y) \in \ell \times \ell$, where $\ell := [x, y]$.
\end{proof}

\begin{proof}[Proof of Lemma \ref{unique direction}]
Let $x \in \Gamma_{i, q, m}$, and suppose $e_1, e_2$ are two distinct directions such that $[[x - e_i, x+ e_i]]$ is an oriented stretch ray. Then we have
$$f(x+ \frac{e_1}{m}) - f(x - \frac{e_2}{m}) = \frac{|e_1| + |e_2|}{m} = \frac{2}{m}$$
while 
$$\text{dist} \left  (x + \frac{e_1}{m}, x - \frac{e_2}{m} \right ) < \frac{2}{m},$$
violating the $1$-Lipschitz property of $f$.
\end{proof}

\begin{proof}[Proof of Lemma \ref{nonempty intersection}]
Let $\ell = [a, b]$ be a stretch ray. Without loss of generality $[[a, b]]$ is its oriented version. Pick by density some hyperplane $\mathbf H_i$ with unit normal $\nu_i$ satisfying $\langle \nu_i, \frac{b - a}{|b-a|}\rangle  := w > 0.$
Let $m := \frac{a + b}{2}$ be the midpoint of $\ell$, and $\pi(m)$ its projection onto $H_i$. 

Then the translate $\mathbf H_i + m - \pi(m)$ intersects $\ell$ uniquely at $m$. Choosing some rational $q$ close to $|m - \pi(m)|$, say within $\delta > 0$, an elementary computation shows that the translate $\mathbf H_{i, q}$ intersects $[a, b]$ uniquely at a point $p$ with $|p-m| < \frac{\delta}{w}$, whenever $\delta > 0$ is small enough.

Thus if we choose $\delta \leq \frac{w|a-b|}{6}$, we obtain that $[p - \frac{b-a}{3}, p + \frac{b-a}{3} ] \subseteq [a, b]$, so that $[  [p - \frac{|a-b|}{3}, p + \frac{|a-b|}{3} ] ]$ is an oriented stretch ray. Thus $\Gamma_{i,q, 3/|a-b|} \cap \ell$ is nonempty as promised.
\end{proof}

\begin{proof}[Proof of Lemma \ref{contains every stretch}]
Let $\ell$ be a stretch ray. By Lemma \ref{nonempty intersection}, there is some $\Gamma_{i,q,m}$ with $\Gamma_{i,q,m} \cap \ell$ nonempty. Let $x$ be such a point in the intersection. Then choosing $k$ larger than the length of $\ell$, we have that $\psi_{i,q,m, k} (\{x\} \times [-k, k]^2)$ contains $\ell \times \ell$ as promised.
\end{proof}

\begin{proof}[Proof of Lemma \ref{lipschitz angle}]
Let $x, y \in \Gamma_{i, q, m}$ be arbitrary, and denote by $e(x)$ and $e(y)$ the uniquely associated stretch directions by Lemma \ref{unique direction}, where we suppress the dependence of $e$ on $i,q, m$ for clarity. By the stretch ray property, we have
\begin{equation}\label{equation1}f(x+\frac{e(x)}{m} ) - f(y - \frac{e(y)}{m} ) = f(x) - f(y) + \frac{2}{m} \geq \frac{2}{m} - |x-y|.\end{equation}

Let us denote by $d_{S^{n-1}}$ the intrinsic geodesic distance on $S^{n-1}$, and note that
\begin{equation}\label{geodesic}d_{S^{n-1}}(e(x), -e(y)) + d_{S^{n-1}}(e(x), e(y)) = \pi\end{equation}
as can be seen by considering a great circle containing $-e(y), e(y)$ and $e(x)$, which shows that $e(x)$ lies on a geodesic connecting $e(y)$ and $-e(y)$, from which the identity (\ref{geodesic}) follows.

Then we have
 \begin{align}  \label{equation2}\text{dist}\left (x+\frac{e(x)}{m}, y - \frac{e(y)}{m}\right ) & \leq \text{dist}\left(x+\frac{e(x)}{m}, x - \frac{e(y)}{m}\right ) + \text{dist} \left (x - \frac{e(y)}{m}, y - \frac{e(y)}{m} \right ) \notag \\
& =  \frac{1}{m}\left ( \sqrt{2 - 2  \cos ( d_{S^{n-1}}(e(x), -e(y)) }  \right )+ |x-y| \notag \\
& = \frac{1}{m} \sqrt {4 - 2  \sin ( d_{S^{n-1}}(e(x), e(y)) } + |x-y|\end{align}
where in the second line we have applied the law of cosines to the triangle with vertices $x+\frac{e(x)}{m}$, $x-\frac{e(y)}{m}$ and $x$, and in the last line we have used the elementary trigonometric identity $cos(\theta) = \sin(\pi-\theta) - 1$, and the identity (\ref{geodesic}).

By the $1$-Lipschitz property of $f$, we have
\begin{equation}\label{1lip}\left |f(x+\frac{e(x)}{m}  ) - f (y - \frac{e(y)}{m} ) \right | \leq \text{dist}\left(x + \frac{e(x)}{m}, y - \frac{e(y)}{m} \right )\end{equation}
so upon substituting equations (\ref{equation1}) and (\ref{equation2}), rearranging terms and applying the triangle inequality we obtain
$$\sin ( d_{S^{n-1}}(e(x), e(y))  \leq 4m|x-y| .$$
Noting that $0 \leq d_{S^{n-1}}(e(x), -e(y))\leq \pi$, by the Taylor expansion of the sine function, we have for all small enough $\delta > 0$, uniformly over all $x, y$ such that $|x - y| \leq \delta$
$$\text{dist}\left (d_{S^{n-1}}(e(x), e(y)), \{0, \mathbb \pi\}\right ) \leq 4m(|x-y| + o_\delta (|x-y|)) \leq 5m |x-y|,$$
where the last inequality holds once $\delta$ is chosen smaller than some absolute constant. As such, we either have
$$d_{S^{n-1}}(e(x), e(y)) \leq 5m|x-y|$$
or 
\begin{equation}\label{latter}|d_{S^{n-1}}(e(x), e(y)) - \pi| \leq 5m|x-y|.\end{equation}
We show that the the latter case (\ref{latter}) cannot happen if $|x-y|$ is small enough. Indeed, in the latter case, from equation (\ref{geodesic}) we would have
$$d_{S^{n-1}} (e(x), -e(y)) \leq 5m|x-y|.$$
Using that the geodesic distance $d_{S^{n-1}} (e(x), -e(y))$ and ambient Euclidean distance $\text{dist}(e(x), e(y))$ are comparable whenever the former distance is small, we obtain that
\begin{align*}\text{dist}(e(x), -e(y)) &\leq 2d_{S^{n-1}} (e(x), -e(y)) \\& \leq 10m|x-y|\end{align*}
for all small enough values of $|x-y|$. From equation (\ref{1lip}), we have
\begin{align*}\frac{2}{m} - |x-y| & \leq \text{dist}(x, y) + \frac{1}{m}\text{dist}(e(x),- e(y)) \\&\leq 11|x-y|\end{align*}
so the the latter case (\ref{latter}) can happen only if $|x-y| \geq \frac{1}{6m}$.

We have thus proven that
$$d_{S^{n-1}}(e(x), e(y)) \leq 5m|x-y|$$
uniformly over all $x, y \in \Gamma_{i, q ,m}$ with $|x-y| < \delta_m$, where the constant $\delta_m$ depends only on $m$. Using that the ambient Euclidean distance is always less than the intrinsic distance on $S^{n-1}$, we obtain
$$|e(x) - e(y)| \leq 5m|x-y|$$
for all $|x-y| <\delta_m.$
To complete this to a global Lipschitz bound, we use the trivial bound $|e(x) - e(y)| \leq 2$, from which we conclude that $e$ is globally Lipschitz with Lipschitz constant $L := \max(5m, \frac{2}{\delta_m})$.
\end{proof}

\section{Sharpness of Theorem \ref{support theorem}}
In this brief section, we give a simple example showing that the Hausdorff dimension upper bound of $n+1$ on the support of optimal plans cannot be improved in general, even for measures with absolutely continuous densities.

Let $\mu := \mathcal L^n \mres [0, 1]^n$, the restriction of Lebesgue measure $\mathcal L^n$ to the unit cube, and let $\nu := \mathcal L^n \mres ([0, 1]^{n-1} \times [-1, 0])$. We claim that there is an optimal plan between $\mu$ and $\nu$ for the $1$-Wasserstein problem (\ref{W1}) whose support has Hausdorff dimension $n+1$. 

Let $\Pi: \mathbb R^n \to \mathbb R^{n-1}$ denote the projection onto the first $n-1$ coordinates, and let
$$\mu = \int_{\mathbb [0, 1]^{n-1}} \mu_x \, d\mathcal L^{n-1}(x)$$
$$\nu = \int_{\mathbb [0, 1]^{n-1}} \nu_x \, d\mathcal L^{n-1}(x)$$
be the disintegration of $\mu, \nu$ respectively with respect to $\Pi$.

As demonstrated in (\cite{Maggi_2023}, Remark 17.4), by Fubini's theorem we can identify 
$$\mu_x =\mathcal H^1 \mres \{x\} \times [0, 1]$$
$$\nu_x = \mathcal H^1 \mres \{x\} \times [-1, 0]$$
for all $x \in [0, 1]^{n-1}$, where $\mathcal H^1$ denotes the Hausdorff measure. 

We take our transport plan $\pi$ to be the relatively independent coupling over the common factor induced by $P$ - that is, we set
$$\pi = \int_{\mathbb [0, 1]^{n-1}} \pi_x \, d\mathcal L^{n-1}(x),$$
where for each $x \in [0, 1]^{n-1}$,
$$\pi_x := \mu_x \otimes \nu_x$$
is the product measure between $\mu_x$ and $\nu_x$.

The $1$-Lipschitz Kantorovich potential $f(x) = x_n$ on $\mathbb R^n$ certifies that $\pi$ is optimal for the $1$-Wasserstein distance, and the support of $\pi$ is easily checked to be 
$$\supp(\pi) = \{(x, t, x, s) \in \mathbb R^{2n}\, | \, x \in [0,1]^{n-1}, t \in [0, 1], s \in [-1, 0]\}$$
which is bi-Lipschitz equivalent to $\mathbb [0,1]^{n+1}$ via the map $(x,t,x,s) \to (x, t, -s)$, and so has Hausdorff dimension $n+1$ as claimed.
\section{Further Questions}
We conclude with some natural further directions for research. We note that although the example in Section 4 shows that the upper bound cannot be improved in general, it does not preclude the existence of optimal transport plans with lower dimensional supports. It is natural to ask whether under sufficient regularity of the source and target measures there exists at least some lower dimensional optimal transport plan. We thus pose:

\textbf{Question 1:} Let $\mu = f  \, \mathcal  L^n, \nu = g \, \mathcal  L^n$ be absolutely continuous probability measures on $\mathbb R^n$ with smooth, compactly supported densities $f, g$. Does there exist an optimal transport plan $\pi$ with $\dim_H (\supp(\pi)) \leq n$?

Noting the key role of Theorem \ref{stretch set theorem}, we are led to ask the following natural generalization, which is itself of intrinsic interest but may also be of some future value in the theory of vector valued optimal transport.

\textbf{Question 2:} Let $f: \mathbb R^n \to \mathbb R^m$ be a Lipschitz map, for $n > m$. Is it true that the stretch set of $f$ is $n+m$ rectifiable?
\printbibliography
\end{document}